\documentclass[a4paper,12pt]{amsart}

\usepackage[english]{babel}
\usepackage{amsthm, latexsym, gastex, amssymb,url,listings}
\usepackage[mathcal]{eucal}
\usepackage{graphicx,url}

\copyrightinfo{2011}{Simone Ugolini}

\DeclareMathOperator{\Tr}{Tr}

\DeclareMathOperator{\Gr}{Gr}

\renewcommand{\phi}[0]{\varphi}
\renewcommand{\theta}[0]{\vartheta}
\renewcommand{\epsilon}[0]{\varepsilon}

\newcommand{\Pro}{\text{$\mathbf{P}^1$}}
\newcommand{\F}{\text{$\mathbf{F}$}}

\newtheorem{theorem}{Theorem}[section]
\newtheorem{lemma}[theorem]{Lemma}

\theoremstyle{definition}

\theoremstyle{remark}
\newtheorem{remark}[theorem]{Remark}

\numberwithin{equation}{section}

\begin{document}

\bibliographystyle{amsalpha}

\date{}

\title[]
{Sequences of binary irreducible polynomials}

\author{S.~Ugolini}
\email{sugolini@gmail.com} 

\begin{abstract} 
In this paper we construct an infinite sequence of binary irreducible polynomials starting from any irreducible polynomial $f_0 \in \F_2 [x]$. If $f_0$ is of degree $n = 2^l \cdot m$, where $m$ is odd and $l$ is a non-negative integer, after an initial finite sequence of polynomials $f_0, f_1, \dots, f_{s}$, with $s \leq l+3$, the degree of $f_{i+1}$ is twice the degree of $f_i$ for any $i \geq s$.  
\end{abstract}

\maketitle
\section{Introduction}
Constructing binary irreducible polynomials of arbitrary large degree is of fundamental importance in many applications. If $f$ is a binary polynomial, namely $f \in \F_2 [x]$ where $\F_2$ is the field with two elements, and has degree $n$, then its $Q$-transform is the polynomial $f^Q (x) = x^n \cdot f(x+x^{-1})$ of degree $2n$.
The $Q$-transform of $f$ is a self-reciprocal polynomial, namely $f^Q$ is equal to its reciprocal polynomial (we remind that, if $g$ is a polynomial of degree $d$, then its reciprocal polynomial is $g^*(x) = x^d \cdot g({x}^{-1})$).

In \cite{var} the following result was proved. 
\begin{theorem}\label{var}
Let $f(x) = x^n + \dots + a_1 x + 1$ be an irreducible polynomial of $\F_2[x]$. Then $f^Q (x)$ is irreducible if and only if $a_1 =1$. 
\end{theorem}
Later Meyn proved in \cite{mey} the following result.
\begin{theorem}
If $f(x) =x^n +a_{n-1} x^{n-1} + \dots + a_1 x + 1$ is an irreducible polynomial of degree $n$ over $\F_2$ such that $a_{n-1} = a_1 = 1$, then $f^Q (x) = x^{2n} + b_{2n-1} x^{2n-1} + \dots + b_1 x + 1$ is a self-reciprocal irreducible polynomial of $\F_2 [x]$ of degree $2n$. 
\end{theorem} 
We classify any irreducible polynomial $f(x) = x^n + a_{n-1} x^{n-1} + \dots + a_1 x + a_0 \in \F_2 [x]$ as follows. 
\begin{itemize}
\item $f$ is of type $(A, n)$ if $a_{n-1} = a_{1} = 1$.
\item $f$ is of type $(B, n)$ if $a_{n-1} = 0$ and  $a_{1} = 1$.
\item $f$ is of type $(C, n)$ if $a_{n-1} = 1$ and $a_1 = 0$.
\item $f$ is of type $(D, n)$ if $a_{n-1} = a_1 = 0$.
\end{itemize} 

If $n$ is any positive integer, then we can write $n = 2^l \cdot m$, for some odd integer $m$ and non-negative integer $l$. As pointed out by Meyn, if $f_0$ is of type $(A,n)$, then by means of repeated applications of the $Q$-transform we can construct an infinite sequence of irreducible polynomials setting $f_{i+1} := f_i^Q$, for $i \geq 0$. We notice that, for any $i$, the degree of $f_{i+1}$ is twice the degree of $f_i$. 

One can wonder what happens if $f_0$ is not of type $(A,n)$. In this paper we will prove that, if $f_0$ is of type $(B, n), (C,n )$ or $(D,n)$, then it is possible to construct an infinite sequence $\{f_i \}_{i \geq 0}$ of binary irreducible polynomials such that, after an initial finite sequence $f_0$, $f_1$, $ \dots,$ $f_s$ with $s \leq l+3$, for $i \geq s$ the degree of $f_{i+1}$ is twice the degree of $f_i$ (see the Subsections \ref{procedure} and \ref{conclusions} ). To prove this fact we will rely upon the properties of the graphs associated with the map $\theta(x) = x+x^{-1}$ over finite fields of characteristic two (see \cite{SU2}) and some properties of the $Q$-transform.

\section{Background}
Given a positive integer $n$, it is possible to construct a graph associated with the map $\theta$ over $\Pro (\F_{2^n}) = \F_{2^n} \cup \{\infty \}$, being $\F_{2^n}$ the field with $2^n$ elements. If $\alpha \in \Pro (\F_{2^n})$, then
\begin{displaymath}
\theta (\alpha) = \left\{
\begin{array}{lll}
\infty & \text{if $\alpha = 0$ or $\infty$}\\
\alpha + \alpha^{-1} & \textrm{otherwise}.
\end{array}
\right.
\end{displaymath}
The vertices of the graph are labelled by the elements of $\Pro (\F_{2^n})$ and an arrow connects the vertex $\alpha$ to the vertex $\beta$ if $\beta = \theta(\alpha)$. We will denote the graph constructed in such a way by $\Gr_n$. 
The elements of $\Pro (\F_{2^n})$ can be partitioned in two sets,
\begin{eqnarray*}
A_n &  = & \{ \alpha \in \F_{2^n}^* : \Tr_{n} (\alpha) = \Tr_{n} (\alpha^{-1}) \} \cup \{ 0, \infty \} \\
B_n & = & \{ \alpha \in \F_{2^n}^* : \Tr_{n} ( \alpha ) \not = \Tr_{n} (\alpha^{-1})  \},
\end{eqnarray*} 
where $\Tr_n (\alpha) = \displaystyle\sum_{i=0}^{n-1} \alpha^{2^i}$ denotes the absolute trace of $\alpha \in \F_{2^n}$.
We notice that, if $f(x) = x^n + a_{n-1} x^{n-1} + \dots + a_1 x + a_0 \in \F_2 [x]$ is the minimal polynomial of an element $\alpha \in \F_{2^n}$, then $a_{n-1} = \Tr_n (\alpha)$. Moreover, if $n > 1$, then $a_1 = \Tr_n (\alpha^{-1})$.  

In Section 4 of \cite{SU2} the structure of the graph $\Gr_n$ is analysed in depth. If $\alpha$ is an element of $\Pro (\F_{2^n})$ such that $\theta^k (\alpha) = \alpha$ for some positive integer $k$, then $\alpha$ is $\theta$-periodic. In the case an element $\alpha$ is not $\theta$-periodic, it is preperiodic, namely some iterate of $\alpha$ is $\theta$-periodic.

Here we summarize the properties of the graph $\Gr_n$ we will use later.
\begin{itemize}
\item Two or zero arrows enter a vertex. 
\item The elements belonging to a connected component of the graph are all in $A_n$ or all in $B_n$. 
\item A connected component having elements in $A_n$ is formed by a cycle, whose vertices are roots of reversed binary trees of depth $l+2$, where $2^l$ is the greatest power of $2$ which divides $n$. Moreover, any vertex belonging to a non-zero level smaller than $l+2$ of a tree has exactly two children.
\end{itemize}

\section{Construction of sequences of binary irreducible polynomials}
The following, which is a special case of Lemma 4 in \cite{mey}, holds.

\begin{lemma}\label{meyn2}
If $f$ is an irreducible polynomial of $\F_2 [x]$ of degree $n > 1$ then either $f^Q$ is a self-reciprocal irreducible polynomial of degree $2n$ or $f^Q$ is the product of a reciprocal pair of irreducible polynomials of degree $n$ which are not self-reciprocal.
\end{lemma}

We prove some technical lemmas.

\begin{lemma}\label{g_splits}
Suppose that $g(x) = x^{2n} + b_{2n-1} x^{2n-1} + \dots + b_1 x + b_0$ is a self-reciprocal binary polynomial of degree $2n$ which factors as $g(x) = f(x) \cdot f^* (x)$, where $f(x) = x^n + a_{n-1} x^{n-1} + \dots + a_1 x + a_0$ is irreducible in $\F_2 [x]$. Then one of the following holds:
\begin{itemize}
\item if $b_{2n-1} = b_1 = 0$, then $a_{n-1} = a_1$;
\item if $b_{2n-1} = b_1 = 1$, then $f$ is of type $(B,n)$ and $f^*$ is of type $(C,n)$ or viceversa.
\end{itemize}
\end{lemma}
\begin{proof}
Expanding the product $f(x) \cdot f^* (x)$ we get a polynomial whose coefficients of the terms of degree $1$ and $2n-1$ are both equal to $a_{n-1} + a_1$. The thesis follows equating the coefficients of the terms of degrees $2n-1$ and $1$ of $f(x) \cdot f^* (x)$ to $b_{2n-1}$ and $b_1$ respectively.
\end{proof}

\begin{lemma}\label{Q_n-1_1}
Let $f(x)=x^n+a_{n-1} x^{n-1} + \dots + a_1 x+ a_0$ be an irreducible binary polynomial of degree $n$ and $g(x) = f^Q (x) = x^{2n}+ b_{2n-1} x^{2n-1} + \dots + b_1 x + b_0$ its $Q$-transform. The following holds.
\begin{itemize}
\item If $a_{n-1} =1$, then $b_{2n-1} = b_1 = 1$.
\item If $a_{n-1} =0$, then $b_{2n-1} = b_1 = 0$.
\end{itemize}
\end{lemma}
\begin{proof}
We notice that, if $k$ is a positive integer, then in the expansion of the term $(x+x^{-1})^k$ only monomials $x^e$, with $-k \leq e \leq k$, appear. Therefore, the coefficients $b_{2n-1}$ and $b_1$ of degrees $2n-1$ and $1$ of $g$ are affected only by the expansions of $(x+x^{-1})^n$ and $(x+x^{-1})^{n-1}$.

If $a_{n-1} = 1$, then
\begin{eqnarray*}
g(x) & = & x^n \cdot \left[(x+x^{-1})^n + (x+x^{-1})^{n-1} + \dots + a_0 \right].
\end{eqnarray*} 
Hence, $g(x)  =  x^{2n} + x^{2n-1} + \dots + x + 1$, namely $b_{2n-1} = b_1 =1$.

Viceversa, if  $a_{n-1} = 0$, then $g(x)  =  x^{2n} + 0 \cdot x^{2n-1} + \dots + 0 \cdot x + 1$, 
namely $b_{2n-1} = b_1 = 0$.
\end{proof}

\begin{lemma}\label{root_Q}
If $f$ is a binary irreducible polynomial of degree $n$ with a root $\alpha \in \F_{2^{n}}$ and $\theta(\beta) = \alpha$ for some $\beta \in \F_{2^{2n}}$, then $\beta$ is root of $f^Q$.
\end{lemma}
\begin{proof}
Since $f(\alpha) = 0$, then $f^Q (\beta) = \beta^n \cdot f(\beta + {\beta}^{-1}) = \beta^n  \cdot f(\alpha) = 0$.
\end{proof}

Now we prove some theorems which relate the types of a polynomial and its $Q$-transform.

\begin{theorem}\label{thm_B}
If $f$ is a polynomial of type $(B,n)$, for some integer $n$ greater than $1$, then $f^Q$ is a polynomial of type $(D,2n)$. If is of type $(B,1)$, then $f(x)=x$ and $f^Q$ splits as $f^Q (x) = (x+1)^2$. 
\end{theorem}
\begin{proof}
If $n>1$, then by Theorem \ref{var} the polynomial $f^Q$ is irreducible of degree $2n$. The conclusion follows from Lemma \ref{Q_n-1_1}.

If $n =1$, then $f(x) = x$ and $f^Q(x) = x^2+1 = (x+1)^2$.
\end{proof}

\begin{theorem}\label{thm_C}
If $f$ is a polynomial of type $(C,n)$, then $f^Q$ can be factored into the product of a reciprocal pair of distinct irreducible polynomials $g_1, g_2$ of degree $n$. Up to renaming, $g_1$ is of type $(B,n)$, while $g_2$ is of type $(C,n)$. 
\end{theorem}
\begin{proof}
By Theorem \ref{var} the polynomial $f^Q$ is not irreducible. Hence it splits into the product of a pair of reciprocal irreducible polynomials, $g_1$ and $g_2$, of degree $n$. If $f^Q(x) = x^{2n}+ b_{2n-1} x^{2n-1} + \dots + b_1 x +b_0$, then, by Lemma \ref{Q_n-1_1}, $b_{2n-1} = b_1 = 1$. Since $f^Q(x) = g_1(x) \cdot g_2(x)$, equating the coefficients we get the thesis.
\end{proof}
\begin{remark}\label{thm_C_remark}
As a consequence of Theorem  \ref{thm_C}, if $f$ is a polynomial of type $(C,n)$, then one of the irreducible factors of $f^Q$, which has been called $g_1$, is of type $(B,n)$. Then, by Theorem \ref{thm_B}, $g_1^Q$ is a polynomial of type $(D,2n)$. Summing all up, starting from a polynomial $f$ of type $(C,n)$, we have constructed a polynomial of type $(D,2n)$.
\end{remark}

\begin{theorem}\label{thm_D}
Let $f$ be a polynomial of type $(D,n)$. Then, $f^Q$ splits into the product of a reciprocal pair of distinct irreducible polynomials $g_1, g_2$ which are both of type $(A,n)$ or both of type $(D,n)$. Moreover, at least one of them is the minimal polynomial of an element $\beta \in \F_{2^n}$ which is not $\theta$-periodic.
\end{theorem}

\begin{proof}
We notice that the roots of $f$ belong to $A_n$, because $f$ is of type $(D,n)$. Moreover, by Theorem \ref{var} and Lemma \ref{meyn2}, $f^Q$ splits into the product of a reciprocal pair of distinct irreducible polynomials $g_1, g_2$. Since $f$ is of type $(D,n)$, by Lemma \ref{Q_n-1_1} and Lemma \ref{g_splits} the polynomials $g_1$ and $g_2$ are both of type $(D,n)$ or both of type $(A,n)$.
Consider a root $ \beta  \in \F_{2^n}$ of the polynomial $g_1$. Then, $\beta^{-1}$ is root of $g_2$ and $\alpha = \theta (\beta)$ is root of $f$.
If $\alpha$ is not $\theta$-periodic, then both $\beta$ and $\beta^{-1}$ are not $\theta$-periodic. On the contrary, if $\alpha$ is $\theta$-periodic, then one between $\beta$ and $\beta^{-1}$, say $\beta$, belongs to the first level of the tree rooted at $\alpha$. All considered, in any case $\beta$ is not $\theta$-periodic.
\end{proof}

\subsection{An iterative procedure.}\label{procedure}
Let $n'$ be a positive integer and $2^{l'}$, for some $l' \geq 0$, the greatest power of $2$ dividing $n'$.
The following iterative procedure takes as input a binary irreducible polynomial $p_0$ of type $(D,n')$ having a root $\alpha \in \F_{2^{n'}}$ which is not $\theta$-periodic and produces a finite sequence of polynomials $p_0, \dots, p_{s'}$, where $p_{s'}$ is of type $(A,n')$. 
After setting $i:=0$, we proceed as follows. 
\begin{enumerate}
\item Factor $p_i^Q$ into the product of two irreducible polynomials $g_1, g_2$ of degree $n'$.
\item Set $i:=i+1$ and $p_i := g_1$.
\item If $p_i$ is of type $(A,n')$, then break the iteration. Otherwise, since $p_i$ is of type $(D,n')$, go to step (1).
\end{enumerate}

The procedure breaks if and only if for some $i$ the polynomial $p_i$ is of type $(A,n')$. Actually, this is the case, as stated below.

\begin{theorem}\label{iteration_breaks}
Using the notations introduced above, there exists a positive integer $s' \leq l'+1$ such that $p_{s'}$ is a polynomial of type $(A,n')$. In addition, if $n'=2m'$ and $p_0 = g^Q$, for some polynomial of type $(B,m')$, then $s' \leq l'$.
\end{theorem}

\begin{proof}
The element $\alpha$ belongs to a positive level $j$ of the reversed binary tree rooted at some element of $A_{n'}$. We want to prove that, for any $0 \leq k \leq l'+2-j$, the polynomial $p_k$ is the minimal polynomial (of type $(A,n')$ or $(D,n')$) of an element $\gamma \in \F_{2^{n'}}$ such that $\theta^k (\gamma) = \alpha$. This is true if $k=0$. Therefore, consider $k < l'+2-j$ and suppose that the assertion is true for $p_k$. This means that $p_k$ is the minimal polynomial of an element $\gamma \in \F_{2^{n'}}$ such that $\theta^k (\gamma) = \alpha$. Since  $\gamma$ is not a leaf of the tree, then there exists an element $\gamma' \in \F_{2^{n'}}$ such that $\theta (\gamma') = \gamma$. By Lemma \ref{root_Q} the element $\gamma'$ is root of $p_k^Q$. Hence $p_k^Q$ splits into the product of a reciprocal pair of distinct irreducible polynomials and $p_{k+1}$ is equal to one of these two polynomials. Moreover, $\gamma'$ or $(\gamma')^{-1}$ is a root of $p_{k+1}$.

Finally, consider $p_{d}$, where $d= l'+2-j$. This polynomial has a root $\gamma \in \F_{2^{n'}}$ such that $\theta^d (\gamma) = \alpha$. We want to prove that $p_d$ is of type $(A,n')$. Suppose on the contrary that it is of type $(D,n')$. Then, $p_d^Q$ splits into the product of two irreducible polynomials of degree $n'$. One of these polynomials has a root $\gamma' \in \F_{2^{n'}}$ such that $\theta (\gamma') = \gamma$. Hence $\theta^{d+1} (\gamma') = \alpha$. Since $\alpha$ belongs to the level $j$ of the tree, then $\gamma'$ belongs to the level $l'+3$ of the tree. But this is not possible, because the tree has depth $l'+2$.
Hence, setting $s' := d \leq l+1$, the polynomial $p_{s'}$ is of type $(A,n')$. 

Suppose now that $n' =2 m'$ and $p_0 = g^Q$, for some polynomial $g$ of type $(B,m')$. Then, $p_0$ is self-reciprocal and $\beta = \theta (\alpha) = \theta (\alpha^{-1}) \in \F_{2^{m'}}$ is root of $g$. Moreover, $\beta$ is not $\theta$-periodic, since its only two preimages with respect to the map $\theta$ are $\alpha, \alpha^{-1}$, which are not $\theta$-periodic. If none of the polynomials $p_0, \dots, p_{l'}$ constructed with the previous iterative procedure is of type $(A,n')$ and $\gamma \in \F_{2^{n'}}$ is a root of $p_{l'}$, then $\alpha$ (root of $p_0$) belongs to some tree of $\Gr_{n'}$ of depth greater than  $l'+2$. Since this is an absurd, we deduce  that in this case $s' \leq l'$.  
\end{proof}

\subsection{Conclusions.}\label{conclusions}

Using the results of Subsection \ref{procedure}, if $f_0$ is any irreducible polynomial of degree $n$ over $\F_2$ and $2^l$ is the greatest power of $2$ dividing $n$, then we can construct an infinite sequence of binary irreducible polynomials. In particular, the initial segment $f_0, \dots, f_{s}$, with $s \leq l+3$, is constructed as explained below. 
\begin{itemize}
\item  If $f_0$ is of type $(A,n)$, then $s=0$.
\item  If $f_0$ is of type $(B,n)$, with $n >1$, then, by Theorem \ref{thm_B}, $f_0^Q$ is of type $(D,2n)$. Set $p_0 := f_0^Q, n':=2 n$ and $l' := l+1 $. Then, by Theorem \ref{iteration_breaks}, it is possible to construct a finite sequence $p_0, \dots, p_{s'}$, where $s' \leq l' = l+1$ and $p_{s'}$ is of type $(A,2n)$. Setting $f_i := p_{i-1}$, for $1 \leq i \leq s'+1 \leq l+2$, we are done.

If $f_0$ is of type $(B,1)$, then $f_0^Q (x) = x^2+1 = (x+1)^2$. In this case we set $f_1 (x) = x+1$ and notice that $f_1$ is of type $(A,1)$. Then, $s = 1$. 
\item If $f_0$ is of type $(C,n)$, then, proceeding as explained in Remark \ref{thm_C_remark}, it is possible to construct a polynomial $f_1$ of type $(B,n)$ and then another polynomial $f_2$ of type $(D,2n)$.  Set $p_0 := f_2, n':=2 n$ and $l' := l+1 $. Then, by Theorem \ref{iteration_breaks}, it is possible to construct a finite sequence $p_0, \dots, p_{s'}$, where $s' \leq l' = l+1$ and $p_{s'}$ is of type $(A,2n)$. Setting $f_i := p_{i-2}$, for $2 \leq i \leq s'+2 \leq l+3$, we are done.
\item If $f_0$ is of type $(D,n)$ and $f_0^Q$ does not split into the product of two polynomials of type $(A,n)$, then one of the two irreducible factors $h_1, h_2$ of $f_0^Q$ is a polynomial of type $(D,n)$ having a root $\alpha \in \F_{2^n}$ which is not $\theta$-periodic. A priori we do not know which of these two polynomials has such a root. Firstly, we set $n':=n$, $l':=l$ and $p_0 := h_1$. Iterating the procedure  described in the previous Subsection we construct a sequence of polynomials $p_0, p_1, \dots, p_{s'}$, where  $s' \leq l'+1$. If none of the polynomials $p_i$ is of type $(A, n)$, then we break the iterations,  set $p_0:= h_2$ and construct a new sequence. This new sequence $p_0, p_1, \dots, p_{s'}$ ends with a polynomial of type $(A,n)$, as proved in Theorem \ref{iteration_breaks}. Setting $f_i := p_{i-1}$ for $1 \leq i \leq s'+1 \leq l+2$ we are done.   
\end{itemize}
In all cases the polynomial $f_s$ is of type $(A,n)$ or $(A,2n)$. We can inductively construct all other terms of the sequence setting $f_{i+1} := f_i^Q$ for $i \geq s$. It is worth noting that for any $i$ the degree of $f_{i+1}$ is twice the degree of $f_i$. Moreover, if $n$ is odd, namely $l=0$, then $s \leq 3$.

\bibliography{Refs}
\end{document}